\documentclass[12pt]{amsart}
\usepackage{amsmath}
\usepackage{amssymb}
\usepackage{amsthm}
\usepackage{array}
\usepackage{xy}
\usepackage[pdftex]{graphicx}
\usepackage{hyperref}
\usepackage{color}
\usepackage{transparent}
\usepackage{latexsym}
\usepackage{ulem}
\usepackage[margin=1.5in]{geometry}
\usepackage{amsmath,scalefnt}
\usepackage{amsthm}
\usepackage{amssymb}
\usepackage{graphicx}
\usepackage{setspace}
\usepackage{hyperref}

\numberwithin{equation}{section}

\newtheorem{thm}{Theorem}[section]
\newtheorem{cor}[thm]{Corollary}
\newtheorem{lem}[thm]{Lemma}

\newtheorem{prop}[thm]{Proposition}
\newtheorem{rem}[thm]{Remark}

\theoremstyle{plain}

\theoremstyle{plain}

\theoremstyle{definition}

\theoremstyle{remark}

\numberwithin{theorem}{section}
\numberwithin{equation}{section}
\numberwithin{figure}{section}

\begin{document}
\title[Boundary regularity and non-transversal intersection]{Boundary regularity and non-transversal intersection for the fully nonlinear obstacle problem}

\author[Emanuel Indrei]{Emanuel Indrei}


\def\signei{\bigskip\begin{center} {\sc Emanuel Indrei\par\vspace{3mm}Department of Mathematics\\  
Purdue University\\
West Lafayette, IN 47907, USA\\
email:} {\tt eindrei@purdue.edu}
\end{center}}


\makeatletter
\def\blfootnote{\xdef\@thefnmark{}\@footnotetext}
\makeatother

\date{}

\maketitle

\begin{abstract}
In this paper non-transversal intersection of the free and fixed boundary is shown to hold in any dimension for obstacle problems generated by fully nonlinear uniformly elliptic operators. Moreover, $C^1$ regularity results of the free boundary are obtained and a classification of blow-up solutions is given.
\end{abstract}

\section{Introduction} 

The dynamics of the free boundary are considered for strong $L^n$-solutions of the following PDE 

\begin{equation}  \label{eqF}
\begin{cases}
F(D^{2}u)=\chi_\Omega & \text{a.e. in }B_{1}^{+}\\
u=0 & \text{on }B'_{1}
\end{cases}
\end{equation}
\vskip .2in 
\noindent where $u \in W^{2,n}(B_{1}^{+})$, $F$ is a convex $C^1$ fully nonlinear uniformly elliptic operator, $\Omega$ is an open set and the free boundary is $\Gamma=\partial \Omega \cap B_1^+$. Under the structural assumptions on $F$, $u \in W^{2,p}(B_{1}^{+})$ for all $p<\infty$, and $u$ satisfies \eqref{eqF} in the viscosity sense \cite{MR1376656}. It was recently shown in \cite{MR3513142} that $u \in C^{1,1}(B_{1/2}^{+})$, see also \cite{MR3542613} for the interior case. The class of bounded solutions is denoted by $P_1^+(0,M, \Omega),$ where $||u||_{L^\infty(B_1^+)} \le M$. In what follows, the tangential touch problem is considered for $$\Omega = \big(\{u \ne 0\} \cup \{\nabla u \neq 0\} \big) \cap \{x_n>0\}\subset \mathbb{R}_+^n.$$  

It has been conjectured that the free boundary intersects the fixed boundary non-transversally and in two dimensions this was proved in \cite{MR3513142} (partial results have also been obtained in \cite{MR2065018}). The case of the Laplacian was treated in \cite{MR1950478, MR1359745}. In this paper, the following theorem is established. 

\begin{thm} (Non-transversal intersection, \S \ref{blowup}) \label{tt}
There exists $r_0>0$ and a modulus of continuity $\omega$ such that 
$$\Gamma(u) \cap B_{r_0}^+ \subset \{x: x_n \le \omega(|x'|)|x'|\}$$ for all $u \in P_1^+(0,M, \Omega)$ provided $0 \in \overline{\Gamma(u)}$.  
\end{thm}

As a consequence, if the solution is non-negative, then the free boundary is $C^1$ in a neighborhood of the origin.

\begin{thm} (Regularity, \S \ref{regularity}) \label{c1r}
Let $u \in P_1^+(0, M, \Omega)$ be non-negative and $0 \in \overline{\Gamma(u)}$. There exists $r_0>0$ such that $\Gamma$ is the graph of a $C^1$ function in $B_{r_0}^+$.
\end{thm}

Moreover, the methods developed to prove the theorems lead to a classification of blow-up limits 
$$\lim_{r\rightarrow 0^+} \frac{u(rx)}{r^2}$$
which in the interior case was carried out in \cite{MR1745013}. 

\begin{thm} (Uniqueness of Blow-Ups, \S \ref{blowup}) \label{cul}
Suppose $u \in P_1^+(0,M, \Omega)$. If $0 \in \overline{\{u \neq 0\}}$ and $\nabla u(0)=0$, then the blow-up limit of $u$ at the origin has the form $$u_0(x)=ax_1x_n+bx_n^2$$ for $a, b \in \mathbb{R}$.
\end{thm}

A similar regularity result holds also in the two-phase case provided that the coincidence set satisfies a certain density assumption: given a set $E$, let $MD(E)$ denote the smallest distance between two parallel hyperplanes containing $E$. To measure thickness of the coincidence set corresponding to a solution $u \in P_1^+(0,M, \Omega)$, let 
$$\delta_r(u,x)=\frac{MD\big((B_1^+ \setminus \Omega)\cap B_r(x)\big)}{r}.$$ If there exists $\epsilon_0>0$ such that $$\delta_r(u,x^0) \ge \epsilon_0$$ for all $r>0$ and $x^0 \in \Gamma \cap B_1^+$, then the free boundary is $C^1$, see Corollary \ref{reg}. In particular, Lipschitz free boundaries are $C^1$ and this is optimal in the sense that in general the free boundary is not $C^{1, \text{Dini}}$ \cite{MR1392033, MR1698524, MR1950478}.

The regularity of the free boundary in the classical obstacle problem has been an area of intense research. If the solution is non-negative, it was shown in \cite{MR0454350} that under a density condition on the coincidence set, the free boundary is $C^1$ and higher regularity follows from \cite{MR0440187}; in general, there exist singularities \cite{MR0516201} and the structure of the singular set appeared in \cite{MR1658612}. Note that there is no density assumption in Theorem \ref{c1r}. It is of interest to understand the most general conditions on the operator, $\Omega$, and boundary data to which this extends.

In \cite{MR3357696, MR3198649}, the authors employed a novel method to handle the fully nonlinear uniformly elliptic case based on a harmonic analysis technique which appeared in \cite{MR2999297} where the authors proved optimal regularity for the solution of the no-sign obstacle problem under the weakest possible assumptions. This was further developed in \cite{MR3542613, MR3625855} to obtain sharp regularity results for more general equations and also improve on existing results for the classical semilinear equation.

The analysis is developed in such a way as to consider the most general configurations. In the above description, solutions with zero Dirichlet boundary data on the hyperplane were considered. If this is not the case, the free boundary may approach the fixed boundary at an angle \cite{MR2281197}.  Several variations of the classical tangential touch problem in recent years have appeared in \cite{MR2874960, MR2727672, MR2267752, MR2237208, MR2142064}. It is of interest to study the largest function space for which uniform results hold.  

The $C^1$ regularity proved in this paper for the case when $u \ge 0$ is natural when considering the historical aspect of the problem where the solution represents the pressure in a liquid: consider water which penetrates a porous medium and divides it into a wet and dry part separated by an interface. The geometry of this interface subject to various boundary conditions was studied in \cite{MR693780}. However, mathematically, the no-sign case is more delicate since one has to understand different phases. For instance, blow-up sequences producing limits of the form $ax_1x_n+bx_n^2$ can be excluded in the one phase case.         

The idea of the proof for tangential touch is to understand the configuration of a blow-up solution which is not a half-space solution and connect it with the interior of the $0$-level set of $u$: if $\partial (\text{int} \{u=0\})$ intersects any ball around the origin, then all blow-up solutions must be half-space solutions. In \cite{MR3513142}, the dimensional constraint is a necessary component of the proof since it relies on the fact that if $$u_j \rightarrow u_0=ax_1x_2+bx_2^2,$$ as $j \rightarrow \infty$, then $$|\nabla u_j|>0$$ in $B_r^+ \setminus B_\delta^+$ for $\delta<r$ and $j$ large; in higher dimensions this is not the case: e.g. consider for $t>0$, $z_t=(0,t,0)$ so that $$\nabla u_0(z_t)=0.$$ 

In the subsequent pages, the following idea is developed to circumvent this difficulty: given any $r>0$, take any cylinder oriented in the $x_1$-direction in $B_r^+$. Then there exist points in the cylinder whose $x_1$-coordinate is less than $-r/2$, such that $$|\nabla u_j(x)| \ge c$$ for $j$ large, where the constant is independent of $j$ and of the cylinder. As a result, if there exist elements of $\partial (\text{int} \{u=0\})$ inside $B_r^+$, then one can prove a monotonicity property to obtain information about the growth of the function in the $x_1$-direction inside this cylinder:
suppose there exist non-negative constants $\epsilon, C$ such that 
$$C \partial_e u -u \ge -\epsilon$$ in $B_r^+$, then $$C \partial_e u - u \ge 0$$ in $B_{\frac{r}{2}}^+$, provided $\epsilon$ is small enough (see Lemma \ref{m}). 

The analysis results in the following statement: either re-scalings converge to half-space solutions, or a certain regularity property holds, and this leads to the classification of blow-up limits. For the regularity, the idea is to show that if the solution is non-negative and there is contact between the fixed and free boundary, then the solution is close to a half-space solution. Thereafter, it follows that it is Lipschitz away from the origin, and hence $C^1$ by interior results and since tangential touch holds, the free normal converges to $e_n$.

\section{Preliminaries} \label{pre}
In what follows, $F$ is assumed to satisfy the following structural conditions.
\begin{itemize}
\item $F(0)=0$.
\item $F$ is uniformly elliptic with ellipticity constants $\lambda_{0}$, $\lambda_{1}>0$
such that
$$
\mathcal{P}^{-}(M-N)\le F(M)-F(N)\le\mathcal{P}^{+}(M-N),
$$
where $M$ and $N$ are symmetric matrices and $\mathcal{P}^{\pm}$
are the Pucci operators
$$
\mathcal{P}^{-}(M):=\inf_{\lambda_{0} \le N\le\lambda_{1}} \text{tr}(NM),\qquad\mathcal{P}^{+}(M):=\sup_{\lambda_{0}\le N\le\lambda_{1}}\text{tr} (NM).
$$
\item $F$ is convex and $C^1$.
\end{itemize}

Let $\Omega$ be an open set. A continuous function $u$ belongs to $P_r^+(0,M, \Omega)$ if $u$ satisfies in the viscosity sense:\\

\noindent 1. $F(D^2 u)=\chi_\Omega$ a.e. in $B_r^+$;\\
2. $||u||_{L^\infty(B_r^+)} \le M$;\\
3. $u=0$ on $\{x_n=0\} \cap \overline{B_1^+}=:B'_{1}$.\\

In \cite{MR3513142} it was shown that $W^{2,p}$ solutions are $C^{1,1}$. 
Furthermore, given $u \in P_r^+(0,M, \Omega)$, the free boundary is denoted by 
$
\Gamma=\partial \Omega \cap B_r^+
$ 
and  
$
\Gamma_i=\partial (\text{int} \{u=0\}) \cap B_r^+.
$
A cylinder with respect to the $e_1$-axis is denoted by $$S_{(\alpha, \beta)}(e_1)=\{(x_1,x'',x_n): (x_n-\beta)^2+|x''|^2<\alpha^2\}.$$

\section{Non-transversal intersection and blow-ups}
In this section, minimal assumptions are made on the set $\Omega$ in order to allow for general configurations.
\subsection{Technical tools} 
The following lemma is similar to the interior case \cite{MR1745013, MR3198649} and provides an improvement of monotonicity.   
\begin{lem} \label{m}
Let $u \in P_r^+(0,M, \Omega)$ where $\{u \neq 0\} \subset \Omega$. Let $e \in \mathbb{S}^{n-2} \cap e_n^{\perp}$ and suppose there exist non-negative constants $\epsilon_0, C_0$ such that $C_0 \partial_e u -u \ge -\epsilon_0$ in $B_r^+$. Then there exists $c=c(n, \Lambda, r)>0$ such that if $\epsilon_0 \le c$, then $C_0 \partial_e u - u \ge 0$ in $B_{\frac{r}{2}}^+$. 
\end{lem}

\begin{proof}
By convexity of $F$, there exist measurable uniformly elliptic coefficients $a_{ij}$ such that 
$$ 
F(D^2 u(x+he))-F(D^2 u(x)) \ge a_{ij}(\partial_{ij} u(x+he)-\partial_{ij} u(x)) 
$$
if $x \in \Omega$ provided $h$ is small enough. Therefore, 

$$0 \ge a_{ij} \partial_{ij} \partial_e u \hskip .1in \text{in $\Omega$}.$$ Convexity also yields 

$$a_{ij} \partial_{ij} u \ge F(D^2 u(x))-F(0) = 1 \hskip .1in \text{in $\Omega$}.$$ Suppose now that there exists $y \in B_{\frac{r}{2}}^+$ for which  $C_0 \partial_e u(y) - u(y) < 0.$ Let $w(x)=C_0 \partial_e u(x)-u(x)+\frac{|x-y|^2}{2n\Lambda}$. Since $\lambda Id \le (a_{ij}) \le \Lambda Id$, it follows by the above that $L w \le 0$ in $\Omega$ where $L=a_{ij}\partial_{ij}$. The maximum principle implies $\min_{\partial(\Omega \cap B_r^+)} w=\min_{\Omega \cap B_r^+} w < 0$. Note that $w\ge 0$ on $\partial \Omega$ and likewise on $\{x_n=0\}$. Therefore, the minimum occurs on $\partial B_r$ and thus $0 > -\epsilon_0+\frac{1}{8n \Lambda}r^2$, a contradiction if $\epsilon_0$ is small enough.  
\end{proof}

\begin{rem} \label{rem1}
One may take $\epsilon_0=c r^2$, where $c>0$ depends only on the dimension and ellipticity constants of $F$.  
\end{rem}

\begin{rem}
If $u \ge 0$, then $\partial_{e_n} u \ge 0$ on $\{x_n=0\} \cap B_r$ and Lemma \ref{m} holds therefore in this case for all $e \in \mathbb{S}^{n-1}$ such that $e \cdot e_n \ge 0$. 
\end{rem}

The next two lemmas highlight properties of the blow-up candidates. 
\begin{lem} \label{c}
Let $u_0(x)=ax_1x_n+bx_n^2$ with $a \neq 0$ and $R\ge1$. Then there exists $c=c(a,b)>0$ such that $$\inf_D |\nabla u_0(x)| \ge c,$$ where $D=\{x=(x_1,x'',x_n): R > |x| > R/2, |x''| \le \delta(R)\}$ for some $\delta(R)>0$.   
\end{lem}

\begin{proof}
Note $|\nabla u_0(x)|^2=a^2x_n^2+a^2x_1^2+2abx_1x_n+4b^2x_n^2$ so that if $|x_n|>\frac{1}{3}$, then $|\nabla u_0(x)|^2 \ge \frac{a^2}{9}$. If $|x_n| \le \frac{1}{3}$, then for points that satisfy $|x''| \le \sqrt{\frac{5}{72}R}$, where $x''=(x_2,x_3,\ldots,x_{n-1})$, it follows that 
$$
x_1^2 > \frac{23}{72} R^2.
$$
If $b \neq 0$, let $\epsilon^2 \in (\frac{1}{a^2+4b^2}, \frac{1}{b^2}).$ Then 
\begin{align*}
|\nabla u_0(x)|^2&\ge(a^2+4b^2-\frac{1}{\epsilon^2})x_n^2+(a^2-\epsilon^2a^2b^2)x_1^2\\
&> (a^2-\epsilon^2a^2b^2)(\frac{23}{72} R^2).
\end{align*} 
\end{proof}
%

\begin{lem} \label{d}
Let $u_0(x)=ax_1x_n+bx_n^2$ with $a > 0$ and $R\ge 1$. Then there exists $C_0=C_0(a,b,R)>0$ such that 
$$
C_0\partial_{x_1} u_0(x)-u_0(x) \ge 0
$$
in $B_R^+$.
\end{lem}

\begin{proof}
The condition is equivalent to $ax_n(C_0-x_1) \ge b x_n^2$. Since $x_1 \le R$ and $0\le x_n \le R$, it follows that any $C_0 \ge \frac{b}{a}R+R$ satisfies the condition.
\end{proof}

A non-uniform version of the next result was shown in \cite{MR3513142}. It consists of the following alternative, either all re-scalings yield half space solutions or one may construct a specific sequence which produces a limit having the form in Lemma \ref{c}. The main point here is that this procedure can be applied to \textit{blow-up limits of $\{u_j\}$} $\subset P_1^+(0,M,\Omega)$, i.e. limits of the form $$\lim_{k \rightarrow \infty} \frac{u_{j_k}(s_kx)}{s_k^2},$$ where $\{j_k\}$ is a subsequence of $\{j\}$ and $s_k \rightarrow 0^+$.

\begin{prop} \label{th1}
Let $\{u_j\} \subset P_1^+(0,M, \Omega)$ and suppose $\{\nabla u_j \neq 0\} \cap \{x_n>0\} \subset \Omega$, $\nabla u_j(0)=0$. Then one of the following is true:\\
(i) all blow-up limits of $\{u_j\}$ at the origin are of the form $u_0(x)=b x_n^2$ for some $b >0$;\\ 
(ii) there exists a blow-up limit of $\{u_j\}$ of the form $ax_1x_n+bx_n^2$ for $a \neq 0$, $b \in \mathbb{R}$.
\end{prop}

\begin{proof}
Let $$N:= \limsup_{|x|\rightarrow 0, x_n>0} \frac{1}{x_n} \sup_{u \in \{u_j\}} \sup_{e \in \mathbb{S}^{n-2} \cap e_n^{\perp}} \partial_e u(x)$$ and consider a sequence $\{x^k\}_{k \in \mathbb{N}}$ with $x_n^k>0$, $u_{j_k} \in \{u_j\}$, and $e^k \in \mathbb{S}^{n-2} \cap e_n^{\perp}$ 
such that the previous limit is given by $$\lim_{k \rightarrow \infty} \frac{1}{x_n^k}  \partial_{e^k} u_{j_k}(x^k).$$ Note that $N<\infty$ by $C^{1,1}$ regularity for the class $P_1^+(0,M, \Omega)$ and the boundary condition (see \cite{MR3513142}). By compactness, $e^k \rightarrow e_1 \in \mathbb{S}^{n-2}$ (along a subsequence) so that up to a rotation, 
$$N= \lim_{k \rightarrow \infty} \frac{1}{x_n^k}  \partial_{x_1} u_{j_k}(x^k).$$   
\noindent Next, if $$\tilde u_j(x):= \frac{u_{k_j}(s_jx)}{s_j^2} \rightarrow u_0(x)$$ for some sequence $s_j \rightarrow 0^+$, where the convergence is in $C_{loc}^{1,\alpha}(\mathbb{R}_+^n)$ for any $\alpha \in [0,1)$, $u_0 \in C^{1,1}(\mathbb{R}_+^n)$ satisfies the following PDE in the viscosity sense
\begin{equation} \label{m2}
\begin{cases}
F(D^{2}u_0)=1 & \text{a.e. in }\mathbb{R}_+^n\cap\Omega_0\\
|\nabla u_0|=0 & \text{in }\mathbb{R}_+^n\backslash\Omega_0\\
u=0 & \text{on }\mathbb{R}_+^{n-1},
\end{cases}
\end{equation}
where $\Omega_0 =\{\nabla u_0 \neq 0 \} \cap \{x_n>0\}$. Note that 
\begin{equation} \label{part}
N \ge \lim_j \bigg| \frac{\partial_{x_i}  u_{k_j}(s_j x)}{s_j x_n}\bigg| = \lim_j \bigg| \frac{\partial_{x_i} \tilde u_j(x)}{x_n} \bigg| =\bigg|\frac{\partial_{x_i} u_0(x)}{x_n}\bigg|
\end{equation}
for all $i \in \{1,\ldots, n-1\}$. 
If $N=0$, then $\partial_{x_i} u_0=0$ for all $i \in \{1,\ldots,n-1\}$ so that $u_0(x)=u_0(x_n)$ and the conditions readily imply $u_0(x_n)=bx_n^2$. Since $N$ does not depend on the sequence $\{s_j\}$ it follows that in this case all blow-up limits have the previously stated form. Suppose that $N>0$, let $r_k=|x^k|$, and consider the re-scaling of $u_{j_k}$ with respect to $r_k$. Note that along a subsequence, $y^k:=\frac{x^k}{r_k} \rightarrow y \in \mathbb{S}^{n-1}$. By the choice of $r_k$, $$\lim_{k \rightarrow \infty} \frac{v(y^k)}{y_n^k}=\lim_{k \rightarrow \infty} \frac{\partial_{x_1} \tilde u_{k}(y^k)}{y_n^k}=\lim_{k \rightarrow \infty} \frac{\partial_{x_1} u_{j_k}(r_ky^k)}{r_ky_n^k} =N,$$ where $v= \partial_{x_1} u_0$. In particular, $$v(y)=Ny_n$$ and by the argument in \cite{MR3513142}, $u_0(x)=ax_1x_n+bx_n^2$ with $a \neq 0$. 
\end{proof}
\subsection{Theorems \ref{tt} and \ref{cul}} \label{blowup}
In what follows, the technical tools are utilized to prove that either all re-scalings yield half-space solutions or there exists a subsequence of the re-scalings which are classical solutions in a small half-ball around the origin.  
\begin{prop} \label{ke}
Suppose $\{u_j\} \subset P_1^+(0,M, \Omega)$. If $0 \in \overline{\{u_j \neq 0\}}$ and $\nabla u_j(0)=0$, then one of the following is true:\\
(i) all blow-up limits of $\{u_j\}$ at the origin are of the form $u_0(x)=bx_n^2$ for $b>0$;\\
(ii) there exists $\{u_{k_j}\} \subset \{u_j\}$ such that for all $R \ge 1$, there exists $j_R \in \mathbb{N}$ such that for all $j \ge j_R$, $$u_{k_j} \in C^{2,\alpha}(B_{\frac{Rr_{j}}{4}}^+),$$ where the sequence $\{r_{j}\}$ depends on $\{u_j\}$.  
\end{prop}

\begin{proof}
Either all blow-up limits are of the form $u_0(x)=bx_n^2$ or there exists a subsequence, say $$\tilde u_j(x)=\frac{u_{k_j}(r_jx)}{r_j^2},$$ producing a limit of the form $u_0(x)=ax_1x_n+bx_n^2$ for $a>0$ (up to a rotation). Let $c=c(a,b)$ be the constant from Lemma \ref{c} and note that since $\tilde u_j \rightarrow u_0$ in $C_{loc}^{1,\alpha}$, there exists $j_0=j_0(a, R) \in \mathbb{N}$ such that for every cylinder $S_{(\alpha, \beta)}(e_1)$ there exists $x \in S_{(\alpha, \beta)}(e_1) \cap B_R^+$ such that $|\nabla \tilde u_j(x)| \ge \frac{c}{2}$ for all $j \ge j_0$, where $R \ge 1$. Now choose a constant  $C_0=C_0(a,b, R)>0$ (in fact, one may select any constant $C_0 \ge R(\frac{b}{a}+1)$) such that $$C_0 \partial_{x_1} u_0 -u_0 \ge 0$$ in $B_R^+$ and $j_0' \ge j_0$ for which 
\begin{equation} \label{mon}
C_0 \partial_{x_1} \tilde u_j -  \tilde u_j \ge 0 \hskip .2 in \text{in $B_{\frac{R}{2}}^+$}
\end{equation} 
whenever $j \ge j_0'$ by Lemma \ref{m}. Now fix $j \ge j_0'$ and suppose $z \in \Gamma_i(\tilde u_j) \cap B_{\frac{R}{2}}^+$. Then there exists a ball $B \subset \text{int} \{\tilde u_j=0\} \cap B_{\frac{R}{2}}^+$. Note that this ball generates a cylinder $S$ in the $e_1$- direction. Now select $x \in S \cap B_R^+$ for which $|\nabla \tilde u_j(x)|>0$ and $-R<x_1<-R/2$. In particular, there exists a small ball around $x$, say $\tilde B$ such that $F(D^2 \tilde u_j)=1$ in $\tilde B$ and one may assume $\tilde B \subset \{\tilde u_j \neq 0\}$. Note that $\tilde B$ is contained in the cylinder $S$ and let $E_t=\tilde B+te_1$ for $t \in \mathbb{R}$. If $t>0$ is such that $\overline{E_t} \cap \{\tilde u_j=0\} \neq \emptyset$, and for all $0\le s<t$, $E_s \cap \{\tilde u_j=0\} = \emptyset$, choose $y \in \overline{E_t} \cap \{\tilde u_j=0\}.$ Moreover, note that if $\tilde u_j > 0$ in $\tilde B$, then by \eqref{mon} it follows that $\tilde u_j$ is strictly positive at a point in $\{\tilde u_j=0\}$, a contradiction. Thus $\tilde u_j < 0$ in $\tilde B$. Next, by convexity of $F$
$$a_{kl} \partial_{kl} \tilde u_j \ge 0 \hskip .1in \text{in $E_t$}.$$ Since $0=\tilde u_j(y)>\tilde u_j(x)$ for $x \in E_t$ and $y$ satisfies an interior ball condition, then Hopf's lemma implies that $\frac{\partial}{\partial n} \tilde u_j(y)>0$, where $n$ is the outer normal to the ball at $y$. Now, if there exists $z \in B_\delta(y)$ such that $\tilde u_j(z)>0$, then this contradicts the monotonicity, if $\delta>0$ is sufficiently small: $\overline{E_{\eta}} \subset B \subset int \{\tilde u_j=0\}$ for $\eta>0$ large enough and since $\tilde u_j(z)>0$, the monotonicity \eqref{mon} implies that $\tilde u_j(z+e_1s)>0$, for some $s>0$ such that $z+e_1s \in \{\tilde u_j=0\}$. Hence, $\tilde u_j\le 0$ on $B_\delta(y)$ and thus $\nabla \tilde u_j(y)=0$, a contradiction.      
The conclusion is that for $j \ge j_0'$, $$\Gamma_i(\tilde u_j) \cap B_{\frac{R}{2}}^+=\emptyset.$$ In particular, $(B_{\frac{R}{2}}^+ \setminus \Omega_j)^o=\emptyset$ and non-degeneracy implies that $|B_{\frac{R}{2}}^+\setminus \Omega_j|=0$. Thus the $C^{1,1}$ function $\tilde u_j$ satisfies $F(D^2 \tilde u_j)=1$ in $B_{\frac{R}{2}}^+$ in the viscosity sense and the up to the boundary Evans-Krylov theorem implies that $\tilde u_j \in C^{2,\alpha}(B_{\frac{R}{4}}^+).$ In particular, $u_{k_j} \in C^{2,\alpha}(B_{\frac{Rr_j}{4}}^+).$ 
\end{proof}

\begin{proof}[proof of Theorem \ref{tt}]
If not, then there exists $\epsilon>0$ such that for all $k \in \mathbb{N}$ there exists $u_k \in P_1^+(0,M, \Omega)$ with 
\begin{equation} \label{cont}
\Gamma(u_k) \cap B_{1/k}^+ \cap \mathcal{C}_\epsilon \neq \emptyset, 
\end{equation}
where $0 \in \overline{\Gamma(u_k)}.$ Now we consider two cases. First, suppose all blow-ups of $\{u_k\}$ are half-space solutions. Let $x_k \in \Gamma(u_k) \cap B_{1/k}^+ \cap \mathcal{C}_\epsilon$ and set $y_k=\frac{x_k}{r_k}$ with $r_k=|x_k|$. Consider $\tilde u_k(x)=\frac{u_k(r_kx)}{r_k^2}$ so that $y_k \in \Gamma(\tilde u_k)$, $\tilde u_k \rightarrow bx_n^2$, $y_{k} \rightarrow y \in \partial B_1 \cap C_\epsilon$ (up to a subsequence), and $y \in \Gamma(u_0)$, a contradiction. Second, select a subsequence $\{u_{k_j}\}$ of $\{u_k\}$ such that for all $j \ge j_2$, $u_{k_j} \in C^{2,\alpha}(B_{\frac{r_{j}}{2}}^+)$, where $j_2 \in \mathbb{N}$ and the sequence $\{r_{j}\}$ depends on $\{u_k\}$. Since $0 \in \overline{\Gamma(u_{k_j})},$ there exists $$x_j \in \Gamma(u_{k_j}) \cap B_{\frac{r_{j}}{2}}^+$$ which contradicts the continuity of $F$.   
\end{proof}

\begin{proof}[proof of Theorem \ref{cul}]
By Proposition \ref{ke}, either $u_0(x)=bx_n^2$ or $D^2u(0)$ exists and the rescaling of $u$ is given by $$u_j(x)=\frac{u(r_j x)}{r_j^2}=\langle x, D^2u(0)x\rangle + o(1).$$ Since $u_0(x',0)=0$ for $x' \in \mathbb{R}^{n-1}$, it follows that $u_0$ has the claimed form (up to a rotation). 
\end{proof}

Combining the non-transversal intersection with \cite[Theorem 1.3]{MR3198649}, the following result holds. 

\begin{cor} \label{reg}
Let $u \in P_1^+(0,M, \Omega)$ and $0 \in \overline{\Gamma(u)}$. Suppose that for some $\epsilon_0>0$, $$\delta_r(u,x^0) \ge \epsilon_0$$ for all $r>0$ and $x^0 \in \Gamma \cap B_1^+$. Then there exists $r_0>0$ such that $$\Gamma \cap \overline{B_{r_0}^+}=\{x:x_n=\phi(x')\} \cap \overline{B_{r_0}^+},$$ where $\phi$ is $C^1$. 
\end{cor}

\begin{proof}
Let $x^0 \in \Gamma \cap B_1^+$, then there exists a neighborhood of $x^0$ such that $\Gamma$ is represented as the graph of a $C^1$ function with respect to some coordinate system \cite[Theorem 1.3]{MR3198649}. By Theorem \ref{tt} it follows that the normal to the free boundary converges to $e_n$ so that this function can be taken with respect to the $\{x_n=0\}$ hyperplane. 
\end{proof}

\begin{rem} \label{rema}
If the free boundary can be represented as the graph of a Lipschitz function close to a contact point, then the thickness condition is satisfied. In general, the free boundary is not more regular than $C^{1, \text{Dini}}.$ This is in sharp contradistinction to the interior case.  
\end{rem}


The existence of non-tangential second derivatives follows in a standard way. 

\begin{cor}
Let $u \in P_1^+(0,M, \Omega)$ and $0 \in \overline{\Gamma(u)}$. Then 
$$\displaystyle\lim_{|x|\rightarrow 0}\partial_{ij} u(x)$$ exists non-tangentially for all $i,j \in \{1,\ldots, n\}$.
\end{cor}

\begin{proof}
Suppose $\{x^j\}$ is such that $x_n^j \ge \kappa |(x')^j|$ for some $\kappa>0$. Then, for all $j$ large, $x^j \in \Omega$. Let $$u_j(x)=\frac{u(r_jx)}{r_j^2}$$ where $r_j=|x^j|$ and $y^j=\frac{x^j}{r_j}$ so that along a subsequence, $y^j \rightarrow y \in \mathbb{S}^{n-1}$. There exists $\mu>0$ such that $B_\mu(y) \subset \Omega_j$ for all $j$ large so that $u_j \in C^{2,\alpha}(B_\mu(y))$. Therefore, $$D^2 u_j(y^j)=D^2 u(x^j) \rightarrow D^2 u_0(y).$$ Since $D^2 u_0$ is a constant matrix, it is independent of $y$, and therefore independent of the subsequence. 
\end{proof}

\begin{rem}
In general, one cannot expect the tangential second derivatives to match the non-tangential derivatives. Consider the case when $0 \in \overline{\Gamma_i(u)}$ so that there exists a collection of balls $\{B_\alpha\}$ where $B_\alpha \subset int\{u=0\}$ and on $B_\alpha$, $D^2 u =0$. 
\end{rem}

\section{Regularity} \label{regularity}
Corollary \ref{reg} follows in a standard way once non-transversal intersection is established (via interior regularity). In the physical case when $u \ge 0$, it turns out that one may dispense with density conditions. The key is to exploit the boundary condition and estimate a maximal mixed partial derivative. 
 
\begin{lem} \label{rl}
Suppose $u \in P_1^+(0,M, \Omega)$ and  $u\ge 0$. Then for any $\epsilon>0$ there exists $r(\epsilon, M)>0$ such that if $x^0 \in \Gamma(u) \cap B_{1/2}^+$ and $d=x_n^0<r,$ then 
$$ 
\sup_{B_{2d}^+(x^0)} |u-h| \le \epsilon d^2, \hskip .2in \sup_{B_{2d}^+(x^0)} | \nabla u - \nabla h| \le \epsilon d, 
$$
where $$h(x)= b[(x_n-d)^+]^2,$$ and $b>0$ depends on the ellipticity constants of $F$.
\end{lem}

\begin{proof}
Suppose not, then there exists $\epsilon>0$, non-negative $u_j \in P_1^+(0,M, \Omega)$, and $x^j \in \Gamma(u_j) \cap B_{1/2}^+$ with $d_j=x_n^j \rightarrow 0$, for which 
$$\sup_{B_{2d_j}(x^j)^+} |u_j- b[(x_n-d_j)^+]^2|>\epsilon d_j^2,$$ or 
$$\sup_{B_{2d_j}(x^j)^+} |\nabla u_j- 2b(x_n-d_j)^+|>\epsilon d_j.$$   
Let $\tilde u_j(x)=\frac{u_j((x^j)'+d_jx)}{d_j^2}$ so that in particular 
$$||\tilde u_j - h||_{C^1(B_2^+(e_n))} \ge \epsilon,$$ where $h(x)=b[(x_n-1)^+]^2$. Since $\tilde u_j(e_n)=|\nabla \tilde u_j(e_n)|=0,$ the $C^{1,1}$ regularity of $\tilde u_j$ implies that $|\tilde u_j(x)| \le C|x-e_n|^2$. By passing to a subsequence, if necessary, $$\tilde u_j \rightarrow u_0$$ where $u_0 \in C^{1,1}(\mathbb{R}_+^n)$ satisfies the following PDE in the viscosity sense
\begin{equation} \label{m3}
\begin{cases}
F(D^{2}u_0)=1 & \text{a.e. in }\mathbb{R}_+^n\cap\Omega_0,\\
|\nabla u_0|=0=u_0 & \text{in }\mathbb{R}_+^n\backslash\Omega_0,\\
u_0=0 & \text{on }\mathbb{R}_+^{n-1}.
\end{cases}
\end{equation}   
Now let $$N=\limsup_{|x|\rightarrow 0, x_n>0} \frac{1}{x_n} \sup_{u\in P_1^+\cap \{u \ge 0\}} \sup_{e \in \mathbb{S}^{n-2} \cap e_n^{\perp}} \sup_{y \in \overline{B_{1/2}^+} \cap \{x_n=0\}} \partial_e u(x+y)$$ and note that $N<\infty$ by $C^{1,1}$ regularity and the boundary condition: for any 
$e \in \mathbb{S}^{n-2} \cap e_n^{\perp}$ and $y \in \overline{B_{1/2}^+} \cap \{x_n=0\}$, it follows that $\partial_{e} u(x'+y)=0$. Furthermore, 

\begin{equation} \label{part}
N \ge \lim_j \bigg| \frac{\partial_{x_i}  u_j(d_j x+(x^j)')}{d_j x_n}\bigg| = \lim_j \bigg| \frac{\partial_{x_i} \tilde u_j(x)}{x_n} \bigg| =\bigg|\frac{\partial_{x_i} u_0(x)}{x_n}\bigg|
\end{equation}
for all $i \in \{1,\ldots, n-1\}$. In particular, let $v= \partial_{x_1} u_0$ so that in $\mathbb{R}_+^n$, 
\begin{equation} \label{ine}
|v(x)| \le Nx_n. 
\end{equation}
If $N=0$, then $\partial_{x_i} u_0=0$ for all $i \in \{1,\ldots,n-1\}$ and therefore $u_0(x)=u_0(x_n)$. Since $e_n$ is a free boundary point, it follows that $u_0=h$, a contradiction. Thus $N>0$  
and there is a sequence 
$\{x^k\}_{k \in \mathbb{N}}$ with $x_n^k>0$, $u_k \in P_1^+(0,M, \Omega)$, $u_k \ge 0$, $y^k \in \overline{B_{1/2}^+} \cap \{x_n=0\}$, and $e^k \in \mathbb{S}^{n-2} \cap e_n^{\perp}$ such that $$N=\lim_{k \rightarrow \infty} \frac{1}{x_n^k}  \partial_{e^{k}} u_k(x^k+y^k).$$ 
By compactness, $e^k \rightarrow e_1 \in \mathbb{S}^{n-2}$ (along a subsequence) so that up to a rotation, 
$$N= \lim_{k \rightarrow \infty} \frac{1}{x_n^k}  \partial_{x_1} u_k(x^k+y^k).$$ Let 
$$\tilde u_k(x)=\frac{u_k(y^k+r_kx)}{r_k^2},$$ where $r_k=|x^k|$, $z^k=\frac{x^k}{r_k}$, and note that along a subsequence $z^k \rightarrow z \in \mathbb{S}^{n-1}$ and $\tilde u_k \rightarrow u_0$. It follows that $\partial_{x_1}u_0(z)=Nz_n$ and proceeding as in \cite{MR3513142} one deduces that $u_0(x)=ax_1x_n+cx_n+\tilde b x_n^2$ for $a \neq 0$ and $c, \tilde b \in \mathbb{R}$, contradicting that $u \ge 0$.   
\end{proof}

\begin{proof}[proof of Theorem \ref{c1r}]
By Lemma \ref{rl} it follows that in a neighborhood of the origin, there is a cone of fixed opening that can be placed below and above each free boundary point. This implies that the free boundary is Lipschitz continuous. Away from the origin, it is therefore $C^1$ by interior results. Moreover, since the intersection of $\Gamma$ and the origin occurs non-transversally, the aperture of the cones can be taken arbitrarily close to $\pi$, and this implies that the free boundary is $C^1$ at the origin.  
\end{proof}

\bibliographystyle{alpha}

\bibliography{ngonref2}

\signei

\end{document}